\documentclass[a4paper, 11pt]{amsart}
\usepackage[latin1]{inputenc}
\usepackage[ T1]{fontenc}
\usepackage[english]{babel}
\usepackage{amssymb}
\usepackage{amsmath}
\usepackage{amsthm}
\usepackage{amscd}
\usepackage{amsfonts}
\usepackage{stmaryrd}
\usepackage{pb-diagram}
\usepackage{epic,eepic,epsfig}
\usepackage{a4wide}
\usepackage{nextpage}
\usepackage{fancyhdr}
\usepackage{enumerate}
\usepackage{hyperref}
\usepackage{float}
\usepackage{tikz}

\newtheorem{prop}{Proposition}[section]

\newtheorem{ex}[prop]{Example}
\newtheorem{lemma}[prop]{Lemma}

\newtheorem{theorem}[prop]{Theorem}

\theoremstyle{definition}

\def\BC{\mathbb{C}}

\def\BZ{\mathbb{Z}}

\usepackage[OT2,T1]{fontenc}
\DeclareSymbolFont{cyrletters}{OT2}{wncyr}{m}{n}
\DeclareMathSymbol{\Sha}{\mathalpha}{cyrletters}{"58}

\begin{document}

\title{Partitions and Hecke images}
\author{}
\author{Florian Breuer \and Fabien Pazuki}

\address{School of Information and Physical Sciences, The University of Newcastle,
University Drive, Callaghan, NSW 2308, Australia.}
 \email{Florian.Breuer@newcastle.edu.au}

\address{Department of Mathematical Sciences, University of Copenhagen,
 Universitetsparken 5, 
2100 Copenhagen \O, Denmark.}
 \email{fpazuki@math.ku.dk}

\maketitle

\noindent \textbf{Abstract.}
We obtain a new family of relations satisfied by the partition function.
In contrast with most partition relations, these involve non-trivial roots of unity. We present two proofs, one using the fact that the discriminant modular form is a multiplicative Hecke eigenform, and one direct proof using q-series.
%

{\flushleft
\textbf{Keywords:} Partitions, modular forms, Hecke images, q-series.\\
\textbf{Mathematics Subject Classification:} 11F11, 11P82, 05A17. }


\begin{center}
---------
\end{center}

\thispagestyle{empty}

\maketitle

\section{Introduction}\label{section def}

For a positive integer $n$, we let $p(n)$ denote the number of partitions of $n$, i.e. 
the number of distinct ways of representing $n$ as a sum of positive integers. Two sums with the same terms in a different order are considered to be the same. By convention $p(0)=1$, and then $p(1)=1$, $p(2)=1$, $p(3)=2$, etc.

Let us denote by $P(q)$ the generating $q$-series of this partition function $p(\cdot)$. It satifies 
\begin{equation}
    P(q)=\sum_{n=0}^{\infty} p(n) q^n=\prod_{n=1}^{\infty}(1-q^n)^{-1}.
\end{equation}

Another classical $q$-series is given by the modular discriminant $\Delta(q)$. It may be defined as 
\begin{equation}
    \Delta(q)=q\prod_{n=1}^{\infty}(1-q^n)^{24},
\end{equation}
see for instance Theorem 8.1 page 62 of \cite{Sil94} for the classical link with the discriminant modular form (note that some authors normalize the discriminant modular form by multiplying this $q$-series by a power of $2\pi$). Let us remark that one may define (see page 65 of \cite{Sil94}) the Dedekind eta function by 
$$\eta(\tau)=e^{\frac{2i\pi\tau}{24}}\prod_{n=1}^{\infty}(1-q^n),$$ 
where $\tau\in{\mathcal{H}} := \{z\in\BC \;|\; \mathrm{Im}(z) > 0\}$, and $q=e^{2i\pi\tau}$. It naturally satisfies $\eta(\tau)^{24}=\Delta(\tau)$. For recent properties of the eta function in relation with partition theory, one may refer to \cite{Zag21}.

It is a natural idea to use the direct link between these $q$-expansions to deduce new properties for partitions, see for instance \cite{Wol08} where results on the parity of the partition function are obtained in this manner. 

The discriminant modular form is an important object, which appears notably in the theory of elliptic curves over the complex numbers. The study of transformations of elliptic curves by finite homomorphisms, the isogenies, gave deep consequences in the study of partitions: the action of Hecke operators were used in \cite{Ono00} to prove a conjecture of Erd\H{o}s on the distribution of values of $p(n)$ modulo an integer $m$. Very recently, Ono also found new formulas linking partitions with endomorphisms of elliptic curves and was able to recover classical congruences on partition numbers by a new method, see \cite{Ono25}. 

There are also other results relating $\Delta(\tau)$ to the size of Hecke images, see for instance \cite{Sil90, Aut, BP24, BGP25}. This lead us to realize that there is another way of using isogenies to deduce a result on partitions. In this note, we prove the following result.
For $N\geq 1$, consider the set of matrices: 
\[
C_N=\left\{\left(\begin{matrix}
a & b \\
0 & d
\end{matrix}\right)\; : \; a,b,d\in\BZ, \;
ad=N, \; a\geq1, \; 0\leq b\leq d-1,\; \gcd(a,b,d)=1 \right\}.
\]
The number of elements of $C_N$ is often denoted $\psi(N)$ and we have
\[
\#C_N = \psi(N) = \sum_{\substack{d|N \\ r=\gcd(d, N/d)}}\frac{d\varphi(r)}{r}  = N\prod_{p|N}\left(1+\frac{1}{p}\right). 
\]

\begin{theorem}\label{MainResult}
    Let $N\geq 1$ be a positive integer, $\zeta_N = e^{2\pi i/N}$ a primitive $Nth$ root of unity and $X$ an indeterminate. Then the following formula holds.
    \[
    \prod_{\left(\begin{matrix}
    a & b \\
    0 & d
    \end{matrix}\right) \in C_N}
    \left(1 + \sum_{n=1}^\infty \zeta_N^{abn} p(n) X^{a^2n}\right)
    =
    \left(1 + \sum_{n=1}^\infty p(n)X^{Nn}\right)^{\psi(N)}.
    \]
\end{theorem}

We stress that this result is valid for any integer $N\geq1$. Let us spell out the case $N=2$.

\begin{ex}
    When $N=2$, Theorem \ref{MainResult} gives
    \[
    \left(1+\sum_{n=1}^\infty p(n) X^{4n}\right)
    \left(1+\sum_{n=1}^\infty p(n) X^{n}\right)
    \left(1+\sum_{n=1}^\infty (-1)^n p(n) X^{n}\right)
    =
    \left(1+\sum_{n=1}^\infty p(n) X^{2n}\right)^3.
    \]
\end{ex}

Let us now move to a related result. Similar to $C_N$, we define
\[
C'_N=\left\{\left(\begin{matrix}
a & b \\
0 & d
\end{matrix}\right)\; : \; a,b,d\in\BZ, \;
ad=N, \; a\geq1, \; 0\leq b\leq d-1 \right\},
\]
where we have dropped the gcd condition on the entries. We have a natural bijection
\[
C_N' \stackrel{\sim}{\longrightarrow} \coprod_{f^2|N}C_{N/f^2},
\qquad 
\left(\begin{matrix}
a & b \\
0 & d
\end{matrix}\right) \longmapsto 
\left(\begin{matrix}
a/f & b/f \\
0 & d/f
\end{matrix}\right) \in C_{N/f^2}; 
\quad f=\gcd(a,b,d).
\]
It follows that
\[
\sigma(N) := \sum_{d|N}d = \#C'_N = \sum_{f^2|N}\#C_{N/f^2} = \sum_{f^2|N}\psi(N/f^2)
\]
and by M\"obius inversion,
\[
\psi(N) = \sum_{f^2|N}\mu(f)\sigma(N/f^2).
\]

With this in place, we also prove the following result.

\begin{theorem}\label{MainResult'}
    Let $N\geq 1$ be a positive integer, $\zeta_N = e^{2\pi i/N}$ a primitive $Nth$ root of unity and $X$ an indeterminate. Then the following formula holds.
    \[
    \prod_{\left(\begin{matrix}
    a & b \\
    0 & d
    \end{matrix}\right) \in C_N'}
    \left(1 + \sum_{n=1}^\infty \zeta_N^{abn} p(n) X^{a^2n}\right)
    =
    \left(1 + \sum_{n=1}^\infty p(n)X^{Nn}\right)^{\sigma(N)}.
    \]
\end{theorem}

In the next two sections, we will deduce Theorem \ref{MainResult} from a property of the discriminant function and prove Theorem \ref{MainResult'} using $q$-series manipulations. Since these two theorems are equivalent by M\"obius inversion, this gives two different proofs for each theorem.


\section{Proof of Theorem \ref{MainResult}}


Out first proof relies on the following result, found in Autissier's \cite[Lemme 2.2]{Aut}, and which (in our opinion) deserves to be better known. For other applications of this specific result, see \cite{Aut, BP24, BGP25}. For more recent work on multiplicative Hecke relations, see \cite{KS25}.

\begin{lemma}\label{discriminant}
    Let $N\geq 1$ be a positive integer. The discriminant function satisfies the following ``multiplicative Hecke eigenform'' property. For any $\tau\in{\mathbb{C}}$ with $\mathrm{Im}\tau>0$, one has
    \[
    \prod_{\gamma \in C_N}
    \Delta(\gamma(\tau)) 
    =
    \big(-\Delta(\tau)\big)^{\psi(N)}.
    \]
\end{lemma}

Before starting the proof of Theorem \ref{MainResult}, let us first gather some useful arithmetic results. We denote by
$\gamma=\left(\begin{matrix}
a_\gamma & b_\gamma \\
0 & d_\gamma
\end{matrix}\right)$ the matrices in $C_N$.

\begin{lemma}\label{arithmetic}
     For any integer $N\geq1$, we have
    \begin{equation}\label{agamma}
        \sum_{\gamma\in{C_N}}\frac{a_\gamma}{d_\gamma}=\psi(N).
    \end{equation}
    For any odd integer $N\geq1$, we have
        \begin{equation}\label{bgammaodd}
        \sum_{\gamma\in{C_N}}\frac{b_\gamma}{d_\gamma}=\frac{1}{2}\psi(N).
        \end{equation}
     For any even integer $N\geq1$, we have   \begin{equation}\label{bgammaeven}
     \sum_{\gamma\in{C_N}}\frac{b_\gamma}{d_\gamma}=\frac{1}{2}\psi(N)-2^t,
        \end{equation}
        where $t$ is the number of odd primes dividing $N$.
\end{lemma}

\begin{proof} 
    We start with the equality (\ref{agamma}). We drop the $\gamma$-index in intermediate steps to simplify the notation. $$\sum_{\gamma\in{C_N}}\frac{a_\gamma}{d_\gamma}=\sum_{d\vert N}\frac{a}{d}\sum_{\substack{0\leq b< d\\ \gcd(a,b,d)=1}}1=\sum_{d\vert N}\frac{a}{d}\frac{d\varphi(r)}{r}=\sum_{d\vert N}\frac{a\varphi(r)}{r}=\sum_{a\vert N}\frac{a\varphi(r)}{r}=\psi(N),$$ where $\varphi$ is Euler's function, where $r=\mathrm{gcd}(a,d)$, and where we used $ad=N$.

    Let us now turn to the equality (\ref{bgammaodd}). 
    $$\sum_{\gamma\in{C_N}}\frac{b_\gamma}{d_\gamma}=\sum_{d\vert N}\!\!\!\sum_{\substack{0\leq b< d\\ \gcd(a,b,d)=1}}\frac{b}{d}=\sum_{d\vert N}\!\!\!\sum_{\substack{0\leq b< d\\ \gcd(a,b,d)=1}}\!\!\!\frac{1}{2}\frac{b+(d-b)}{d}=\sum_{d\vert N}\!\!\!\sum_{\substack{0\leq b< d\\ \gcd(b,r)=1}}\frac{1}{2}=\sum_{d\vert N}\frac{d\varphi(r)}{2r}=\frac{1}{2}\psi(N),$$
    where we used the fact that $b\neq d-b$ in the sum, as $N$ is odd (which implies $d$ is odd).

    We now deal with equality (\ref{bgammaeven}). The integer $N$ is now even, hence can be written as $N=2^{e_0}p_1^{e_1}\ldots p_t^{e_t}$, where $p_1,\ldots, p_t$ are pairwise distinct odd primes and $e_i\geq1$ for any index $i\geq0$. We can use the previous computation with the additional correcting term coming from the integers $b$ for which $b=d-b$:

$$\sum_{\gamma\in{C_N}}\frac{b_\gamma}{d_\gamma}=\sum_{d\vert N}\!\!\!\sum_{\substack{0\leq b< d\\ \gcd(a,b,d)=1}}\frac{b}{d}=\frac{1}{2}\psi(N)-\frac{1}{2}\sum_{\substack{d\vert \frac{N}{2}\\ \gcd(a,2d)=1}}1=\frac{1}{2}\psi(N)-2^t.$$
\end{proof}
The interested reader can also check \cite{Bor10} equation (3) for other formulas involving these quantities.
We now turn to the proof of our theorem.

\begin{proof}{(of Theorem \ref{MainResult})}

We start by Lemma \ref{discriminant}, which gives 

\begin{equation}\label{}
    \prod_{\gamma \in C_N}
    \Delta(\gamma(\tau))^{-1}
    =
    \big(-\Delta(\tau)\big)^{-\psi(N)},
\end{equation}
and we have $$\gamma(\tau)=\left(\begin{matrix}
a & b \\
0 & d
\end{matrix}\right)\cdot \tau=\frac{a}{d}\tau+\frac{b}{d},$$
hence 
\begin{equation}\label{24}
    \prod_{\gamma \in C_N}
    e^{-2\pi i\frac{a\tau+b}{d}}\left(\sum_{n=0}^{+\infty} p(n) e^{2\pi i n \frac{a\tau+b}{d}}\right)^{24}
    =
    \left(-e^{-2\pi i\tau}\left(\sum_{n=0}^{+\infty} p(n) e^{2\pi i\tau n}\right)^{24}\right)^{\psi(N)}.
\end{equation}
Now one also has
\begin{equation}\label{exp}
    \prod_{\gamma \in C_N}
    e^{-2\pi i\frac{a\tau+b}{d}}\
    =
    \left(-e^{-2\pi i \tau}\right)^{\psi(N)},
\end{equation}
by a direct application of Lemma \ref{arithmetic} on the sum of exponents obtained by developing the left hand side of \ref{exp}. Now inject (\ref{exp}) into (\ref{24}) and set $X = e^{2\pi i/N} = e^{2\pi i /ad}$, this yields the result, as the expansions in Theorem \ref{MainResult} have constant term equal to $1$.
\end{proof}

\section{Proof of Theorem \ref{MainResult'}}

We propose a proof of Theorem \ref{MainResult'}, and hence another proof of Theorem \ref{MainResult} using a slightly different method. The idea of looking at these specific $q$-series came from our first proof of Theorem \ref{MainResult}.

The strategy here is to obtain the following $q$-series identities, the first of which implies Theorem \ref{MainResult'}, second of which implies Theorem \ref{MainResult}.

\begin{prop}\label{q-identity}
    Let $N\geq 1$ and suppose $|q|<1$. We have
    \begin{enumerate}
    \item $\displaystyle \prod_{n=1}^{\infty}
    \prod_{\left(\begin{matrix}
    a & b \\
    0 & d
    \end{matrix}\right) \in C_N'}\left(
    1-\zeta_d^{bn}q^{Nn/d^2}
    \right) = \prod_{n=1}^\infty 
    \left(1-q^n\right)^{\sigma(N)}$,

    \item $\displaystyle \prod_{n=1}^{\infty}
    \prod_{\left(\begin{matrix}
    a & b \\
    0 & d
    \end{matrix}\right) \in C_N}\left(
    1-\zeta_d^{bn}q^{Nn/d^2}
    \right) = \prod_{n=1}^\infty 
    \left(1-q^n\right)^{\psi(N)}$.
    \end{enumerate}
\end{prop}

\begin{proof}
    It suffices to prove (1), as the second claim then follows by (multiplicative) M\"obius inversion.

    We start by naming the left hand side
    \begin{align*}
       P_{q,N} & := \prod_{n=1}^{\infty}
    \prod_{\left(\begin{matrix}
    a & b \\
    0 & d
    \end{matrix}\right) \in C_N'}\left(
    1-\zeta_d^{bn}q^{Nn/d^2}
    \right). 
    \end{align*}
    Let us now compute
    \begin{align*}
    P_{q,N} & = \prod_{n=1}^\infty \prod_{d|N}\prod_{b=0}^{d-1} 
    \left(
    1-\zeta_d^{bn}q^{Nn/d^2}
    \right) \\
    & = \prod_{n=1}^\infty \prod_{d|N} \prod_{p=0}^{e-1}\prod_{b'=0}^{d'-1}
    \left(
    1-(\zeta_{d'}^{n'})^{b'}q^{Nn'/dd'}
    \right),
    \end{align*}
    where we set
    $e = \gcd(d,n)$ and $d=d'e$, we set $n=n'e$ as well as 
    $b=pd'+b'$ with $0\leq b' < d'$.

    Using the fact that $\zeta_{d'}^{n'}$ is a primitive $d'$-th root of unity, we thus obtain
    \begin{align*}
        P_{q,N} & = \prod_{n=1}^\infty \prod_{d|N} (1-q^{Nn'/d})^e \\
        & = \prod_{e|N}\prod_{d'|\frac{N}{e}}\prod_{\substack{ n' \geq 1 \\ \gcd(n',d')=1}} \left(1-q^{\frac{N}{e}\frac{n'}{d'}}\right)^e \\
        & =  \prod_{e|N}\prod_{n=1}^\infty (1-q^n)^e 
         = \prod_{n=1}^\infty (1-q^n)^{\sigma(N)},
    \end{align*}
    where we have used the fact that, for fixed $e|N$, the exponent $\frac{N}{e}\frac{n'}{d'}$ ranges exactly over all positive integers $n$.
\end{proof}

\section*{Acknowledgements}
The authors thank Abdulaziz Mohammed Alanazi, 
Augustine Munagi, Darlison Nyirenda and Dimbinaina Ralaivaosaona for interesting conversations during the Stellenbosch Number theory conference in January 2025. The authors thank the IRN GandA (CNRS) for support. They also thank Pascal Autissier for interesting conversations during the YuBi60 conference in Bordeaux in June 2025. This work was completed during the Ren\'e 25 conference in August 2025 in Tahiti, and the authors thank the organizers for the hospitality.
FB is supported by the Alexander-von-Humboldt Foundation.
FP is supported by ANR-20-CE40-0003 Jinvariant.

\end{document}